\documentclass[oneside,english]{amsart}
\usepackage[T1]{fontenc}
\usepackage[latin9]{inputenc}
\usepackage{amsthm}
\usepackage{amstext}
\usepackage{amssymb}
\usepackage{graphicx}

\makeatletter
\numberwithin{equation}{section}
\numberwithin{figure}{section}
\theoremstyle{plain}
\newtheorem{thm}{\protect\theoremname}
  \theoremstyle{plain}
  \newtheorem{lem}[thm]{\protect\lemmaname}
  \theoremstyle{remark}
  \newtheorem*{rem*}{\protect\remarkname}
  \theoremstyle{definition}
  \newtheorem*{example*}{\protect\examplename}
  \theoremstyle{definition}
  \newtheorem{example}[thm]{\protect\examplename}

\makeatother

\usepackage{babel}
  \providecommand{\examplename}{Example}
  \providecommand{\lemmaname}{Lemma}
  \providecommand{\remarkname}{Remark}
\providecommand{\theoremname}{Theorem}

\begin{document}
\global\long\def\dome{\text{dom}\mathcal{E}}

\global\long\def\doml{\text{dom}\Delta}

\title{Infinite Propagation Speed for Wave Solutions on Some P.C.F. Fractals}

\author{Yin-Tat Lee}
\begin{abstract}
The finite difference method for the wave equation on p.c.f. fractals
suggests that the propagation speed of the wave equation may be infinite.
We prove this is indeed true if the heat kernel satisfies a sub-Gaussian
lower bound. Furthermore, we provide a sub-Gaussian upper bound for
the solution of the wave equation given the heat kernel sub-Gaussian
upper bound.
\end{abstract}
\maketitle

\section{Introduction}

In \cite{dalrymple1999fractal}, Dalrymple, Strichartz, and Vinson
pointed out that there is no maximum propagation speed on the Sierpinski
Gasket (SG) because of a scaling property of SG. In other words, there
is no $C$ such that for all $x$ and $t>0$, the fundamental solution
of the wave equation at point $x$ and time $t$ is supported in $B_{Ct}(x)$.
However, it does not rule out the possibility that the fundamental
solution is supported in $B_{f(t)}(x)$ for some continuous function
$f$ such that $f(0)=0$. 

In this paper, we first provide an error analysis for the finite difference
method on p.c.f. fractals with regular harmonic structure. Let $u$
be a solution of the wave equation on the fractal $K$, and let $u_{m}$
be the  solution on the level $m$ approximation $V_{m}$ of $K$.
In Theorem \ref{thm:na error}, we show that $u_{m}(x,t)\approx u(x,h_{m}t)$
where $h_{m}$ is a time renormalization factor. Interestingly, the
$h_{m}$ decreases faster than the grid size does as $m$ increases
for most of p.c.f. fractals. It means that the propagation speed of
$u_{m}$ increases as $m$ increases. Although the result will not
be used in the later proof, it gives the heuristic reason why the
infinite speed holds. 

In Theorem \ref{thm:infinite_speed}, we prove the infinite propagation
speed. If the initial position is zero and initial velocity is positive,
then $u$ attains positive values for all points $x\in K$ within
arbitrary small time period. The proof uses a heat kernel lower bound
and a relation of heat equation and wave equation. In Theorem \ref{thm:avg_wave},
we prove a off-diagonal upper bound for the solution of wave equation
using a complex time heat kernel upper bound. This upper bound is
also sub-Gaussian.

\section{Preliminaries}

At first, we define briefly the notations and concepts introduced
by Jun Kigami\cite{kigami1993harmonic}. An iteration function system
(IFS) is a finite set of contraction mappings $\{F_{i}\}_{i=1}^{N}$
on a complete metric space. An IFS fractal $K$ is the unique compact
set such that $K=\bigcup_{k=1}^{N}F_{k}K$. A connected IFS fractal
is called post critical finite (p.c.f.) if there is a finite set $V_{0}$
such that $F_{j}K\cap F_{k}K\subseteq F_{j}V_{0}\cap F_{k}V_{0}$
for $j\neq k$. For a word $\omega=i_{1}i_{2}\dots i_{m}$, we define
$F_{\omega}=F_{i_{1}}\circ F_{i_{2}}\circ\cdots\circ F_{i_{m}}$.
For example, the interval $[0,1]$ is the unique IFS fractal generated
by mappings $\{F_{1}(x)=\frac{1}{2}x,F_{2}(x)=\frac{1}{2}x+\frac{1}{2}\}$
and the corresponding $V_{0}$ is $\{0,1\}$. 

Now we define a sequence of increasing finite graphs $\Gamma_{i}$
to approximate $K$. Let $\Gamma_{0}$ be the complete graph of the
finite set $V_{0}$. For $i>0$, we define $\Gamma_{i}=(V_{i},\ E_{i})$
where $V_{i}=\cup_{k=1}^{N}F_{k}V_{i-1}$ and 
\[
E_{i}=\{(F_{j}x,F_{j}y)\in V_{i}\times V_{i}|\ (x,y)\in E_{i-1}\}.
\]
We define $V_{*}=\cup_{k=1}^{\infty}V_{k}$. Using $[0,1]$ as an
example, the corresponding $V_{n}$ is $\{\frac{i}{2^{n}}|i\in\{0,\cdots,2^{n}\}\}$,
$\Gamma_{n}$ is the simple path with $2^{n}+1$ vertices and $V_{*}$
is the set of dyadic numbers $\{\frac{a}{2^{n}}|n\in\mathbb{N}\ \text{and}\ 0\leq a\leq2^{n}\}$.

For any finite set $V$, a non-negative symmetric bi-linear form $\mathcal{E}$
on $V$ is called a Dirichlet form if $\mathcal{E}(u,u)=0$ for all
constant functions $u$ on $V$ and $\mathcal{E}(u,u)\geq\mathcal{E}([u],[u])$
for any function $u$ on $V$ where $[u]=\min\{\max\{u,0\},1\}$.
For $V'\subset V$, we can induce a Dirichlet form $\mathcal{E}_{V'}$
from $\mathcal{E}_{V}$ by 
\[
\mathcal{E}_{V'}(u,u)=\inf_{v|_{V'}=u}\mathcal{E}_{V}(v,v).
\]
For any IFS fractal $K$, a sequence of Dirichlet form $\{\mathcal{E}_{m}\}$
on $\{V_{m}\}$ is called compatible if $\mathcal{E}_{m}$ is induced
from $\mathcal{E}_{m+1}$ for all $m$. If this sequence satisfies
the equation 
\[
\mathcal{E}_{m+1}(u,v)=\sum_{i=1}^{N}r_{i}^{-1}\mathcal{E}_{m}(u\circ F_{i},v\circ F_{i})
\]
 for some number $r_{i}>0$, we call it a self-similar sequence and
it is said to be regular if $r_{i}<1$. 

If $\mathcal{E}_{m}$ are compatible, $\mathcal{E}_{m}(u|_{V_{m}},u|_{V_{m}})$
is increasing. For any function $u$ on $K$, we define energy $\mathcal{E}(u,u)$
as $\lim_{m\rightarrow\infty}\mathcal{E}_{m}(u|_{V_{m}},u|_{V_{m}})$
and $\dome=\{u\in C(K):\ \mathcal{E}(u,u)<\infty\}$. It is known
that $\dome/\text{constants}$ is a Hilbert space. For any function
$u$ on $V_{m}$, the harmonic extension of $u$ is the unique continuous
function $\tilde{u}$ on $K$ minimizing the energy $\mathcal{E}(\tilde{u},\tilde{u})$.
We define $\psi_{p}^{m}(x)$ to be the harmonic extension of the delta
function $\delta_{xp}$ on $V_{m}$.

For example, we can define a regular self-similar sequence on $[0,1]$
by 
\[
\mathcal{E}_{m}(u,u)=\sum_{k=0}^{2^{m}-1}\left|\frac{u(\frac{k+1}{2^{m}})-u(\frac{k}{2^{m}})}{\frac{1}{2^{m}}}\right|^{2}\frac{1}{2^{m}}.
\]
The corresponding energy $\mathcal{E}(u,u)=\int_{0}^{1}|u'(x)|^{2}dx$
for $u\in C^{1}([0,1])$. The corresponding harmonic extension is
linear interpolation on $V_{m}$; $\psi_{p}^{m}(x)$ is a triangular
function at point $p$ with width $2^{-m+1}$ and $dom\mathcal{E}=H^{1}$.

We define the resistance metric on $K$ by 
\[
R(x,y)=\max\{\mathcal{E}(v,v)^{-1}:\ v(x)=1,\ v(y)=0\}.
\]
It is known that $K$ is compact under resistance metric, in particular,
\begin{equation}
|u(x)-u(y)|^{2}\leq C\mathcal{E}(u,u).\label{eq:morrey inequality}
\end{equation}
where $C=\sup_{x,y\in K}R(x,y)<\infty.$

Next, we define a self-similar probability measure $\mu$ on $K$
by

\[
\mu(A)=\sum_{i=1}^{N}\mu_{i}\mu(F_{i}^{-1}A)
\]
for some $\mu_{i}\in(0,1)$ such that $\sum\mu_{i}=1$. For $u\in\dome$,
the Laplacian of $u$ corresponding to the self-similar $\mu$ is
defined by the weak formulation: $\Delta_{\mu}u=f$ if $f\in L^{2}$
and 
\[
\mathcal{E}(u,v)=-\int fv\ d\mu
\]
for all $v\in\dome$ that vanish on the boundary $V_{0}$. If $\Delta_{\mu}u$
is continuous, we have a pointwise formula for $\Delta_{\mu}u$:
\[
\Delta_{\mu}u(x)=\lim_{m\rightarrow\infty}\mu_{m,x}^{-1}H_{m}u(x)
\]
where $\mu_{m,p}=\int\psi_{p}^{m}(x)\ dx$ and $H_{m}$ is the self-adjoint
matrix such that 
\[
\mathcal{E}_{m}(u,v)=-(u,H_{m}v).
\]

\section{Existence of solutions}

Let $B$ be a finite subset of $V_{*}$. For $u\in C(K)$ vanishing
on $B$, we define the Laplacian $\Delta_{\mu,B}u=f$ if $\mu_{m,x}^{-1}H_{m}u(x)$
converges uniformly to a continuous function $f$ on $K\backslash B$.
\cite[A.2]{kigami2001analysis}

The wave equation with boundary set $B$, initial position $f$ and
initial velocity $g$ is defined by

\begin{equation}
\begin{cases}
\begin{array}{ccll}
u_{tt}(x,t) & = & \Delta_{\mu,B}u(x,t)\quad & (x\in K\backslash B,\ t\in\mathbb{R})\\
u(x,0) & = & f(x) & (x\in K)\\
u_{t}(x,0) & = & g(x) & (x\in K)\\
u(x,t) & = & 0 & (x\in B,\ t\in\mathbb{R})
\end{array} & ,\end{cases}\label{eq:wave_eq}
\end{equation}
where the time derivative $u_{tt}$ is in the classical sense. For
convenience, we write $\Delta$ instead of $\Delta_{\mu,B}$. The
condition $B=\emptyset$ corresponds to Neumann boundary condition
and $B=V_{0}$ corresponds Dirichlet boundary condition. In this paper,
we use $C$ as a generic constant which depends only on the fractal.
Since most of the proofs in this and next sections need extra care
for the case $B=\emptyset$, we omit the proofs for that case in these
two sections. 

By \cite[A.2]{kigami2001analysis}, we have a set of orthogonal eigenvectors
$\left\{ \varphi_{n}\right\} _{n\geq1}$ of $-\Delta$ with corresponding
increasing eigenvalues $\left\{ \lambda_{n}\right\} _{n\geq1}$ such
that $||\varphi_{n}||_{2}=1$ and $\left\{ \varphi_{n}\right\} _{n\geq1}$
spans $\doml$. By the assumption $B\neq\emptyset$, we have $\lambda>0$.
\begin{lem}
\label{lem: sobolev} $\dome=\{\sum a_{n}\varphi_{n}:\sum a_{n}^{2}\lambda_{n}<\infty\}$
and $||u||_{\infty}\leq C\mathcal{E}(u)^{1/2}$.\end{lem}
\begin{proof}
Let $u=\sum a_{n}\varphi_{n}$ such that $M\triangleq\sum a_{n}^{2}\lambda_{n}<\infty$.
Let $u_{m}$ be the partial sum of $u$. Using $\mathcal{E}(u_{m})=\sum_{n=1}^{m}a_{n}^{2}\lambda_{n}<M$
and \eqref{eq:morrey inequality} , we have 
\[
|u_{m}(x)-u_{m}(y)|<C\ M^{1/2}.
\]
Combining with $\int u_{m}^{2}<M/\lambda_{1}$, we get 
\[
||u_{m}||_{\infty}<(C+\lambda_{1}^{-1/2})\ M^{1/2}.
\]
Hence, $u_{m}$ is uniform bounded. Again, by \eqref{eq:morrey inequality},
$u_{m}$ is equicontinuous. So, by the Arzelà\textendash{}Ascoli theorem,
$u\in C(K)$. Since $\dome/\text{constants}$ is a Hilbert space,
we get $\mathcal{E}(u)=\lim\mathcal{E}(u_{m})=M<\infty$. Hence $u\in dom\mathcal{E}$.

The converse follows from $\mathcal{E}(\sum a_{n}\varphi_{n})=\sum a_{n}^{2}\lambda_{n}$.
\end{proof}
For $\doml$, we do not have a similar description using the original
definition. So, we extend the domain of $\Delta$ to $\{\sum a_{n}\varphi_{n}:\ \sum a_{n}^{2}\lambda_{n}^{2}<\infty\}$
by the identity $\Delta(\sum a_{n}\varphi_{n})=-\sum a_{n}\lambda_{n}\varphi_{n}$,
which converges in $L^{2}$. Since $\sum a_{n}\varphi_{n}$ converges
in $L^{\infty}$ and $\varphi_{n}|_{B}=0$, the boundary condition
is satisfied for $u\in\doml$.

Let the initial position be $f=\sum\alpha_{n}\varphi_{n}$ and the
initial velocity be $g=\sum\beta_{n}\varphi_{n}$. We define the formal
solution by

\begin{equation}
u=\sum_{n}\alpha_{n}\cos(\sqrt{\lambda_{n}}t)\varphi_{n}+\sum_{n}\beta_{n}\frac{\sin(\sqrt{\lambda_{n}}t)}{\sqrt{\lambda_{n}}}\varphi_{n}.\label{eq:weaksol}
\end{equation}

It is standard to prove the formal solution is a weak solution under
some condition on $f$ and $g$. In \cite{hu2002nonlinear}, Hu discussed
wave solutions for the Fréchet derivatives. However, in order to complete
the error analysis using the finite difference method, we need to
prove that the classical solution exists.
\begin{thm}
\label{thm: existence}If $f\in\doml$ with $\Delta f\in\dome$ and
$g\in\doml$, then the solution $u$ of the wave equation exists.\end{thm}
\begin{proof}
Let $u$ be the weak solution defined by \eqref{eq:weaksol}. Formally,
we have 
\[
u_{tt}=-\sum_{n}\alpha_{n}\lambda_{n}\cos(\sqrt{\lambda_{n}}t)\varphi_{n}+\sum_{n}\beta_{n}\sqrt{\lambda_{n}}\sin(\sqrt{\lambda_{n}}t)\varphi_{n}.
\]
Fix $x_{0}\in K$ and $\gamma_{n}=\alpha_{n}\lambda_{n}\text{sgn}(\varphi_{n}(x_{0}))$.
Since $\Delta f\in\dome$, we have 
\[
\sum\gamma_{n}^{2}\lambda_{n}=\sum\alpha_{n}^{2}\lambda_{n}^{3}<\infty.
\]
By Lemma \ref{lem: sobolev}, we get 
\begin{eqnarray*}
\sum|\alpha_{n}\lambda_{n}\cos(\sqrt{\lambda_{n}}t)\varphi_{n}(x_{0})| & = & \sum|\alpha_{n}\lambda_{n}\varphi_{n}(x_{0})|\\
 & = & |\sum\gamma_{n}\varphi_{n}(x_{0})|\\
 & \leq & C\ \mathcal{E}(\Delta f)^{1/2}.
\end{eqnarray*}
According to the Weierstrass M-test, $\sum\alpha_{n}\lambda_{n}\cos(\sqrt{\lambda_{n}}t)\varphi_{n}(x_{0})$
converges uniformly for any $t$. Similarly for the term $|\beta_{n}\sqrt{\lambda_{n}}\sin(\sqrt{\lambda_{n}}t)\varphi_{n}(x_{0})|$.
This implies $u_{t}$ and $u_{tt}$ exist in the classical sense.\end{proof}
\begin{rem*}
In \cite{hu2002nonlinear}, Hu used the eigenvalue estimate $\lambda_{n}=O(n^{\alpha})$
to estimate the term $\sum|\alpha_{n}\lambda_{n}\cos(\sqrt{\lambda_{n}}t)\varphi_{n}(x_{0})|$.
That argument requires slightly stronger regularity condition. If
our argument is used to replace all eigenvalue estimates in that paper,
we could arrive the following result:

Let $f$ be a real-valued function on $\mathbb{R}$ satisfying $F(r)\leq C(1+|r|^{2})$
where $F(r)=\int_{0}^{r}f(s)ds$. If $g_{1}\in\doml$ with $\Delta g_{1}\in\dome$,
and $g_{2}\in\doml$, then the nonlinear wave equation with Dirichlet
boundary condition
\[
\begin{cases}
\begin{array}{ccll}
u_{tt}(x,t) & = & \Delta u(x,t)+f(u)\quad & (x\in K\backslash V_{0},\ t\in\mathbb{R})\\
u(x,0) & = & g_{1}(x) & (x\in K)\\
u_{t}(x,0) & = & g_{2}(x) & (x\in K)\\
u(x,t) & = & 0 & (x\in V_{0},\ t\in\mathbb{R})
\end{array}\end{cases}
\]
admits a weak solution, where the second derivative of $u$ is the
Fréchet derivative of $u_{t}$ in $L^{2}$.
\end{rem*}

\section{Finite Difference Method}

The wave equation on $\Gamma_{m}$ is defined by 
\[
u_{m}(x,t+1)=2u_{m}(x,t)-u_{m}(x,t-1)+h_{m}^{2}\mu_{m,x}^{-1}H_{m}u_{m}(x,t)
\]
 where $h_{m}$ is the time span. In this section, we find the difference
between solutions of the wave equation on $K$ and $\Gamma_{m}$. 

First of all, we prove that the wave equation on the approximate graph
is stable.
\begin{lem}
\label{lem:wave bound}Let $V$ be a finite dimension inner product
space. Let $H$ be a positive self-adjoint operator on $V$ with eigenvalues
$\leq3$. Let $\mathcal{E}_{H}(u)=(u,Hu)$, $h$ be a function on
$V\times\mathbb{N}$ and $g$ be a function on $V$. Let $u$ be the
solution of the wave equation

\[
\begin{cases}
\begin{array}{rcl}
u(t+1)-2u(t)+u(t-1) & = & -Hu(t)+h(t)\quad(t\geq1)\\
u(0) & = & 0\\
u(1) & = & g
\end{array} & .\end{cases}
\]
Then we have $\mathcal{E}_{H}(u(t))^{1/2}\leq2(||g||+\sum_{k=1}^{t}||h(k)||)$.\end{lem}
\begin{proof}
Let $\{v_{n}\}$ be the orthonormal eigenvectors of $H$ with corresponding
eigenvalues $\lambda_{n}$. 

For the case $h\equiv0$, let $g(x)=\sum a_{n}v_{n}$. Then the solution
is 
\[
u(x,t)=\sum\alpha_{n}\sin(\theta_{n}t)v_{n}(x)
\]
 where $\theta_{n}=\cos^{-1}(1-\frac{\lambda_{n}}{2})$ and $\alpha_{n}=a_{n}/\sin(\theta_{n})$.
So, the energy at time $t$ is 
\begin{eqnarray*}
\mathcal{E}_{H}(u(\cdot,t)) & = & \sum\alpha_{n}^{2}\lambda_{n}\sin^{2}(\theta_{n}t)\\
 & \leq & \sum\frac{a_{n}^{2}}{\sin(\theta_{n})^{2}}\lambda_{n}\\
 & = & \sum\frac{a_{n}^{2}}{1-\frac{\lambda_{n}}{4}}
\end{eqnarray*}
By the assumption $\lambda_{n}\leq3$, so we have $\mathcal{E}_{H}(u(\cdot,t))\leq4\sum a_{n}^{2}=4||g||^{2}.$

For the general case, let $W_{g}(x,t)$ be the solution of this homogeneous
equation at time $t$ with initial velocity $g$. The result follows
from the formula for the general solution:

\[
u(x,t)=W_{g}(x,t)+\sum_{k=1}^{t}W_{h(\cdot,k)}(x,t-k).
\]

\end{proof}
Next, we estimate the difference between a finite energy function
and its step function approximation. Recall that $\psi_{p}^{m}(x)$
is the harmonic extension of the function $\delta_{xp}$ on $V_{m}$. 
\begin{lem}
\label{lem:Poincare inequality}For $f\in\dome$, we have 
\[
\sum_{x\in V_{m}}\int_{K_{m,x}}|f(x)-f(y)|^{2}dy\leq C\mu_{\max}^{m}r_{\max}^{m}\mathcal{E}(f),
\]
where $\mu_{\max}=\max\mu_{i}$, $r_{\max}=\max r_{i}$ and $K_{m,x}={\rm supp}\psi_{x}^{m}$.\end{lem}
\begin{proof}
Using $|f(y)-f(x)|^{2}\leq C\mathcal{E}(f)$, we have
\[
\int_{K}|f(y)-f(x)|^{2}dy\leq C\ \mathcal{E}(f).
\]
Applying the contraction mappings $F_{\omega}^{-1}$ on both sides,
we get
\begin{eqnarray*}
\int_{F_{\omega}K}|f\circ F_{\omega}^{-1}(y)-f\circ F_{\omega}^{-1}(x)|{}^{2}dy & = & \mu_{\omega}\int_{K}|f(y)-f(x)|^{2}dy\\
 & \leq & C\mu_{\omega}\mathcal{E}(f)\\
 & = & C\mu_{\omega}r_{\omega}\mathcal{E}(f\circ F_{\omega}^{-1}).
\end{eqnarray*}
Thus, for any finite energy function $f$ with support in $F_{\omega}K$,
we have
\[
\int_{K}|f(y)-f(x)|^{2}dy\leq C\ \mu_{\max}^{m}r_{\max}^{m}\mathcal{E}(f)
\]
 where $m=|\omega|$. For $x\in F_{\omega}V_{0}\subset V_{m}$, $K_{m,x}$
is contained in a $m-1$ cell. Thus, 
\[
\int_{K_{m,x}}|f(y)-f(x)|^{2}dy\leq C\ \mu_{\text{max}}^{m-1}r_{\text{max}}^{m-1}\mathcal{E}(f|_{K_{m,x}})
\]
Summing the inequality over $V_{m}$, we have 
\[
\sum_{x\in V_{m}}\int_{K_{m,x}}|f(x)-f(y)|^{2}dy\leq C\ \mu_{\text{max}}^{m-1}r_{\text{max}}^{m-1}\sum_{x\in V_{m}}\mathcal{E}(f|_{K_{m,x}}).
\]
Since $K_{m,x}$ covers $K$ at most $N$ times, $\sum_{x\in V_{m}}\mathcal{E}(f|_{K_{x}})\leq N\ \mathcal{E}(f)$.
\end{proof}
We define $(u,v)_{m}=\sum_{x\in V_{m}}u(x)v(x)\mu_{m,x}$. Under this
inner product, the operator $h_{m}^{2}\mu_{m,x}^{-1}H_{m}$ is self-adjoint.
\begin{lem}
\label{lem:two_norm_difference}For any $f\in\dome$, we have 
\begin{eqnarray*}
\left|||f||_{m}-||f||_{L_{2}}\right| & \leq & C\sqrt{\mu_{\text{max}}^{m}r_{\text{max}}^{m}\mathcal{E}(f)}.
\end{eqnarray*}
\end{lem}
\begin{proof}
By direct calculation, we have
\begin{eqnarray*}
 &  & \left|||f||_{m}^{2}-||f||_{L_{2}}^{2}\right|\\
 & = & \left|\sum_{x\in V_{m}}\int_{K_{m,x}}\left(|f(x)|^{2}\psi_{x}^{m}(y)-|f(y)|^{2}\psi_{x}^{m}(y)\right)dy\right|\\
 & \leq & \sum_{x\in V_{m}}\int_{K_{m,x}}|f(x)-f(y)||f(x)+f(y)|dy\\
 & \leq & \left(\sum_{x\in V_{m}}\int_{K_{m,x}}|f(x)-f(y)|^{2}dy\right)^{1/2}\left(\sum_{x\in V_{m}}\int_{K_{m,x}}|f(x)+f(y)|^{2}dy\right)^{1/2}.
\end{eqnarray*}
Then the result follows from Lemma \ref{lem:Poincare inequality}
and
\begin{eqnarray*}
\sum_{x\in V_{m}}\int_{K_{m,x}}|f(x)+f(y)|^{2}dy & \leq & 2N(||f||_{V_{m}}^{2}+||f||_{L_{2}}^{2})\\
 & \leq & 2N(||f||_{V_{m}}+||f||_{L_{2}})^{2}.
\end{eqnarray*}
\end{proof}
\begin{thm}
\label{thm:na error}Assume $f\in\doml$ with $\Delta f\in\dome$
and $g\in\doml$. Assume both $f$ and $g$ vanish on the boundary
$B$. Assume $B\subset V_{m}$ and eigenvalues of $-h^{2}\mu_{m,x}^{-1}H_{m}$
are $\leq3$. Let $u_{m}$ be the solution of the wave equation on
$\Gamma_{m}:$

\[
\begin{cases}
\begin{array}{ccll}
u_{m}(x,t+1) & = & 2u_{m}(x,t)-u_{m}(x,t-1)+h^{2}\mu_{m,x}^{-1}H_{m}u_{m}(x,t)\quad & (x\in V_{m}\backslash B)\\
u_{m}(x,t) & = & 0 & (x\in B)\\
u_{m}(x,0) & = & f(x) & (x\in V_{m})\\
u_{m}(x,1) & = & f(x)+hg(x)+\frac{h^{2}}{2}\mu_{m,x}^{-1}H_{m}u_{m}(x,0) & (x\in V_{m})
\end{array} & ,\end{cases}
\]
Then, we have
\[
|u_{m}(x,t)-u(x,ht)|\leq Ct(h^{2}+\sqrt{\mu^{m}r^{m}}h)\quad(x\in V_{m},\ t\in\mathbb{N})
\]
where $u$ is the solution of the wave equation on $K$.\end{thm}
\begin{proof}
Assume $g=0$ for simplicity. Let $u=\sum_{n}\alpha_{n}\cos(\sqrt{\lambda_{n}}t)\varphi_{n}$.
By Theorem \ref{thm: existence}, the classical solution $u$ exists
and $\mathcal{E}(\Delta u)<\infty$. The discrete wave equation on
$V_{m}$ comes from discretization of $u_{tt}$ and $\Delta$ as follows:

\begin{eqnarray*}
 &  & u(x,h(t+1))-2u(x,ht)+u(x,h(t-1))\\
 & \approx & h^{2}u_{tt}(x,ht)\\
 & = & h^{2}\Delta u(x,ht)\\
 & \approx & h^{2}\mu_{m.x}^{-1}H_{m}u(x,ht)
\end{eqnarray*}
So we want to estimate the error that appears in those two discretizations. 

For the first error, let 
\[
\text{err}_{1}(x,t)=u(x,h(t+1))-2u(x,ht)+u(x,h(t-1))-h^{2}u_{tt}(x,ht).
\]
We have 
\begin{eqnarray*}
\text{err}_{1}(x,t) & = & 2\sum_{n}\left(\cos(\sqrt{\lambda_{n}}h)-1+\frac{1}{2}\lambda_{n}h^{2}\right)\alpha_{n}\cos(\sqrt{\lambda_{n}}ht)\varphi_{n}(x).
\end{eqnarray*}
Using $|\cos(\sqrt{\lambda_{n}}h)-1+\frac{1}{2}\lambda_{n}h^{2}|<\frac{1}{24}\lambda_{n}^{2}h^{4}$,
we get
\begin{eqnarray*}
||\text{err}_{1}||_{2}^{2} & = & \sum_{n=0}^{\infty}4\left(\cos(\sqrt{\lambda_{n}}h)-1+\frac{1}{2}\lambda_{n}h^{2}\right)^{2}\alpha_{n}^{2}\cos(\sqrt{\lambda_{n}}h_{m}n)^{2}\\
 & \leq & \frac{1}{144}\sum_{n=0}^{N}\lambda_{n}^{4}h^{8}\alpha_{n}^{2}+\sum_{n=N+1}^{\infty}(4+\lambda_{n}h^{2})^{2}\alpha_{n}^{2}\\
 & \leq & \frac{h^{8}\lambda_{N}}{144}\mathcal{E}(\Delta f)+(\frac{32}{\lambda_{N}^{3}}+\frac{2h^{4}}{\lambda_{N}})\mathcal{E}(\Delta f)
\end{eqnarray*}
for any $N$. By \cite[Thm 4.1.5]{kigami2001analysis}, $\lambda_{n}=\Theta(n^{\alpha})$
for some $\alpha>0$. So, we can choose $\lambda_{N}=\Theta(\frac{1}{h^{2}})$.
Thus, $||err_{1}||_{2}^{2}=O(h^{6})$. Similarly, we have $\mathcal{E}(err_{1})=O(h^{4})$.
Using Lemma \ref{lem:two_norm_difference}, we have 
\[
||\text{err}_{1}||_{m}=O(h^{3}+\sqrt{\mu^{m}r^{m}}h^{2}).
\]

For the second error appears in $\Delta\approx\mu_{m.x}^{-1}H_{m}$,
let
\[
\text{err}_{2}(x,ht)=\Delta u(x,ht)-\mu_{m.x}^{-1}H_{m}u(x,ht).
\]
Using $H_{m}u=\int\Delta u\ \psi_{x}^{(m)}d\mu$\cite[A.2.5]{kigami2001analysis},
we obtain 
\begin{eqnarray*}
||\text{err}_{2}||_{m}^{2} & = & \sum_{x}\left|\mu_{m,x}^{-1}\int(\Delta u(x)-\Delta u(y))\ \psi_{x}^{(m)}(y)dy\right|^{2}\mu_{m,x}\\
 & \leq & \sum_{x}\int\left|\Delta u(y)-\Delta u(x)\right|^{2}\psi_{x}^{(m)}(y)dy\\
 & \leq & \sum_{x}\int_{K_{m,x}}\left|\Delta u(y)-\Delta u(x)\right|^{2}dy.
\end{eqnarray*}
Using Lemma \ref{lem:Poincare inequality}, we have 
\[
||\text{err}_{2}||_{m}=O(\sqrt{\mu^{m}r^{m}}).
\]

Let $e(x,t)=u(x,ht)-u_{m}(x,t)$. Then, $e$ satisfies the graph wave
equation: 
\[
e(t+1)-2e(t)+e(t-1)=h^{2}\mu_{m,x}^{-1}H_{m}e(t)+\text{err}_{1}(t)+h_{m}^{2}\text{err}_{2}(t).
\]
Also, we have $e(0)=0$ and $||e(x,1)||_{m}=O(h^{3}+\sqrt{\mu^{m}r^{m}}h^{2})$
by similar estimates. Thus, Lemma \ref{lem:wave bound} implies 
\begin{eqnarray*}
\mathcal{E}_{m}(e)^{1/2} & = & \mathcal{E}_{\mu_{m,x}^{-1}H_{m}}(e)^{1/2}\\
 & = & \frac{\sqrt{2}}{h}\mathcal{E}_{\frac{h^{2}}{2}\mu_{m,x}^{-1}H_{m}}(e)^{1/2}\\
 & = & O(th^{2}+t\sqrt{\mu^{m}r^{m}}h).
\end{eqnarray*}
And the result follows from $||e||_{\infty}=O(\mathcal{E}_{m}(e)^{1/2})$.\end{proof}
\begin{example*}
In Sierpinski Gasket with uniform measure, it is known that \cite[Example 3.7.3]{kigami2001analysis}
\begin{eqnarray*}
\mu_{m,x}^{-1}H_{m}f & = & \frac{3}{2}5^{m}(\frac{1}{\deg(x)}\sum_{x\sim_{m}y}f(y)-f(x))\\
 & \triangleq & \frac{3}{2}5^{m}\Delta_{m}f.
\end{eqnarray*}
 Since $\Delta_{m}f$ is a graph Laplacian, the eigenvalues of $-\Delta_{m}$
are less than or equal to $2$. Since the condition of Theorem \ref{thm:na error}
is satisfied for $h_{m}^{2}\leq5^{-m}$, we take $h_{m}=5^{-m/2}$.
The difference equation becomes
\[
\frac{u(h_{m}(n+1))-2u(h_{m}n)+u(h_{m}(n-1))}{h_{m}^{2}}=\frac{3}{2}5^{m}\Delta_{m}u.
\]
Note that the constant $\sqrt{\frac{3}{2}5^{m}}$ is the scaled propagation
speed. In $[0,1]$, the constant is $2^{m}$, which is the inverse
of the grid size. Thus, the propagation speed in $[0,1]$ is same
for all $m$ but it increases as $m$ increases in SG. And this gives
a heuristic reason that the wave in SG doesn't have finite speed,
which was first observed in \cite{dalrymple1999fractal}.
\end{example*}

\section{Infinite Wave Propagation Speed And Heat Kernel Lower Bound}

In this section, we use heat kernel estimate and a relation between
wave and heat equations to obtain some off-diagonal behaviors for
the wave equation. Since we need Neumann heat kernel estimate, we
assume $B=\emptyset$ in this and next section.
\begin{lem}
\label{lem:heatwave}Assume $f\in\doml$ with $\Delta f\in\dome$
and $g\in\doml$. Let $u$ be the solution of the wave equation. Let
$v(x,t)=\intop_{-\infty}^{\infty}\frac{1}{\sqrt{4t}}\exp(-\frac{s^{2}}{4t})u(x,s)ds$\textup{.
Then $v$ is the solution of heat equation:}
\[
\begin{cases}
\begin{array}{ccll}
v_{t}(x,t) & = & \Delta v(x,t)\quad & (x\in K,\ t\in\mathbb{R})\\
v(x,0) & = & f(x) & (x\in K)
\end{array}.\end{cases}
\]
\end{lem}
\begin{proof}
By theorem \ref{thm: existence}, the classical solution $u$ exists.
Since 
\[
u=\alpha_{1}+\beta_{1}t+\sum_{n=2}^{\infty}\alpha_{n}\cos(\sqrt{\lambda_{n}}t)\varphi_{n}+\sum_{n=2}^{\infty}\beta_{n}\frac{\sin(\sqrt{\lambda_{n}}t)}{\sqrt{\lambda_{n}}}\varphi_{n},
\]
the energy $\mathcal{E}(u(t),u(t))+||u_{t}(t)||_{2}\leq A+Bt^{2}$
for some $A$ and $B$. Thus, $||u(t)||_{\infty}\leq A'+B't$. Since
$N(t,s)\rightarrow0$ rapidly as $s\rightarrow\infty$, $v$ is well-defined
and the result follows by direct verification.
\end{proof}
In \cite{sikora2004riesz}, Adam Sikora proved that for a large class
of self-adjoint operator with Gaussian off-diagonal estimate, the
heat kernel estimates are related to the propagation speed of the
wave equation. For homogeneous hierarchical fractals, we have sub-Gaussian
estimate\cite{barlow1998diffusions} and this makes the propagation
speed infinite.
\begin{thm}
\label{thm:infinite_speed}Suppose the heat kernel satisfies the sub-Gaussian
lower bound:

\[
p(x,y,t)>C\ \exp(-\frac{1}{t^{\beta}})\quad(x,y\in K,\ 1>t>0)
\]
 where $\beta<1$. Assume $f\in\doml$, $\Delta f\in\dome$, $f\geq0$,
$f\neq0$ and $g=0$. Let $u(x,t)$ be the solution of the wave equation.
Then, for all $x\in K$ and $\delta<1$, there is $t<\delta$ such
that $u(x,t)>0$.\end{thm}
\begin{proof}
Let $v(x,t)=\intop_{-\infty}^{\infty}\frac{1}{\sqrt{4t}}\exp(-\frac{s^{2}}{4t})u(x,s)ds$
as defined in Lemma \ref{lem:heatwave}. Since $v$ is the solution
of the heat equation with initial value $f$, for $t<1$, we have
\begin{eqnarray*}
v(x,t) & = & \int p(x,y,t)f(y)dy\\
 & > & C\int\exp(-\frac{1}{t^{\beta}})f(y)dy\\
 & = & C||f||_{1}\exp(-\frac{1}{t^{\beta}}).
\end{eqnarray*}
Take $x$ be any point in $K$. Suppose, on the contrary, $u(x,t)\leq0$
for $t<1$. Since $g=0$, we have $\sup_{t}||u(t)||_{\infty}<A$ for
some $A>0$ and

\begin{eqnarray*}
v(x,t) & \leq & 2\intop_{\delta}^{\infty}\frac{1}{\sqrt{4t}}\exp(-\frac{s^{2}}{4t})u(x,s)ds\\
 & < & 2A\intop_{\delta}^{\infty}\frac{1}{\sqrt{4t}}\exp(-\frac{s^{2}}{4t})ds\\
 & = & 2A\exp(-\frac{\delta^{2}}{4t})\left(\frac{2t}{k}+O(t^{2})\right).
\end{eqnarray*}
It leads to a contradiction that $C||f||_{1}\exp(-\frac{1}{t^{\beta}})<2A\exp(-\frac{\delta^{2}}{4t})\left(\frac{2t}{k}+O(t^{2})\right)$
for $t<\delta$ because $\beta<1$.
\end{proof}
However, the wave oscillates in space instead of spreading, as will
be illustrated by the following example. And this says we cannot expect
$u(x,t)$ to be positive within short times even if $f>1$ and $g=0$.
\begin{example}
\label{example_fra}Consider the Laplacian with Neumann boundary condition
on SG. Using spectral decimation\cite{shima1991eigenvalue,fukushima1992spectral},
we can have the estimate 
\[
-1.5\leq\varphi_{4}\leq2,
\]
where $\varphi_{4}$ is shown in Fig \ref{example_fra}. Now, we define
$f$ by combining copies of $\varphi_{4}$ as shown. On each level,
the solution of the wave equation is of the form $\cos(\sqrt{\lambda}t)\varphi$.
So, the wave oscillates faster on the upper level. Let $\tilde{f}=4f+7$
and $\tilde{u}$ be the wave equation with initial position $\tilde{f}$.
The classical solution exists even though $\mathcal{E}(f)=\infty$.
Although $\tilde{f}\geq1$, $\tilde{u}$ is not positive even in a
short time interval because $\varphi_{4}=2$ at some point.
\end{example}
\begin{center}
\begin{figure}
\hfill{}%
\begin{minipage}[t]{0.4\columnwidth}%
\begin{center}
\includegraphics[width=1\columnwidth]{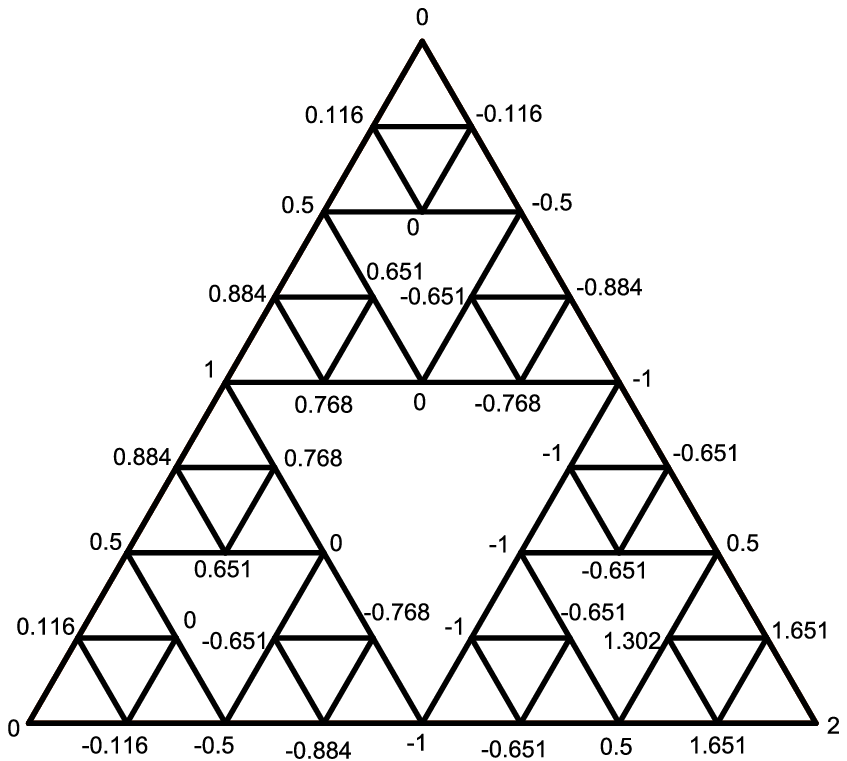}\\
$\varphi_{4}$
\par\end{center}%
\end{minipage}\hfill{}%
\begin{minipage}[t]{0.4\columnwidth}%
\begin{center}
\includegraphics[width=1\columnwidth]{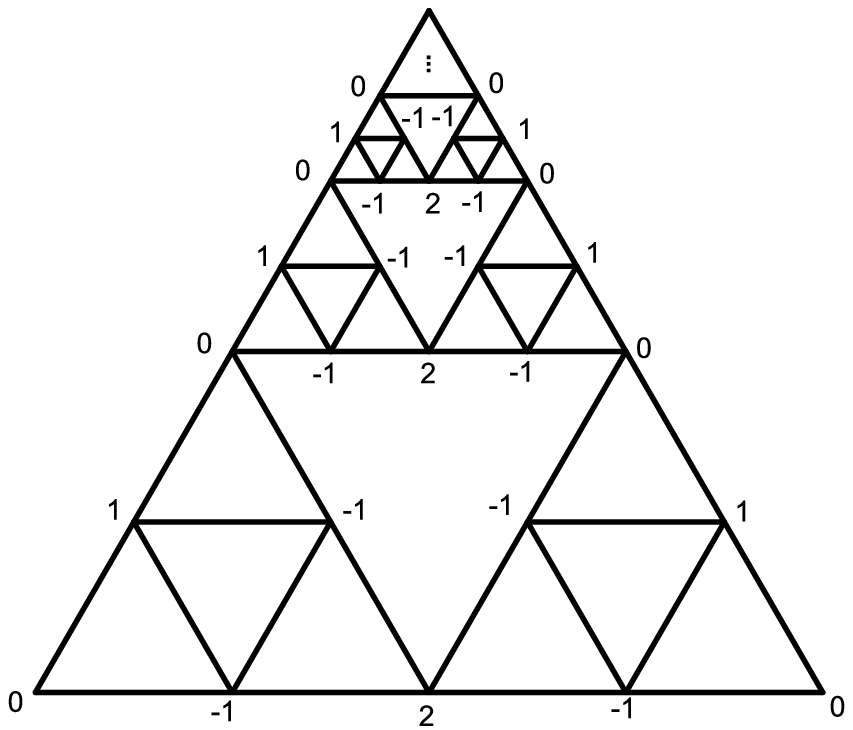}\\
$f$
\par\end{center}%
\end{minipage}\hfill{}

\caption{Functions in Example \ref{example_fra}}
\end{figure}

\par\end{center}

\section{Wave Kernel And Heat Kernel Upper Bound}

Since the wave solution has infinite propagation speed for some fractals,
we would like to get off-diagonal estimates of the solution of the
wave equation for those fractals.

We define $P_{t}u$ be the heat solution with initial data $u$ after
time $t$ where $t>0$, that is, 
\[
P_{t}(\sum a_{n}\varphi_{n})=\sum a_{n}e^{-\lambda_{n}t}\varphi_{n}
\]
where $u=\sum a_{n}\varphi_{n}$. Also, define $W_{t}u$ to be the
solution of the wave equation with initial data $u$ and initial velocity
$0$ after time $t$ where $t>0$, that is, 
\[
W_{t}(\sum a_{n}\varphi_{n})=\sum a_{n}\cos(\sqrt{\lambda_{n}}t)\varphi_{n}.
\]

In this section, we assume the heat equation satisfies the following
kernel upper bound:

\begin{equation}
p(x,y,t)\leq\frac{C}{t^{\alpha}}\exp\left(-C\left(\frac{d(x,y)^{\beta}}{t}\right)^{1/(\beta-1)}\right)\label{eq:heat_assumption}
\end{equation}
for some $\alpha$ and some $\beta>2$ which is true for many fractals
\cite{hambly1999transition,grigor2008off}. 
\begin{lem}
\label{lem:Phragm=0000E9n=002013Lindel=0000F6f}Assume the heat kernel
satisfies the upper bound \eqref{eq:heat_assumption}. For $f\in L^{1}(K)$,
we have
\[
|P_{t+1/z}f(x)|\leq\frac{C}{t^{\alpha}}\exp\left(-Cr^{\frac{\beta}{\beta-1}}z^{\frac{2-\beta}{\beta-1}}\text{Re}z\right)||f||_{1}
\]
for $t>0$ where $r=d(x,\text{supp}f)$.\end{lem}
\begin{proof}
By scaling $f$, we may assume $||f||_{1}=1$. Let $u(z)=P_{t+1/z}f(x)$
which is analytic on $\{{\rm Re}z>0\}$. Note that
\begin{eqnarray*}
p(x,y,z) & = & \sum e^{-\lambda_{n}z}\varphi_{n}(x)\varphi_{n}(y)\\
 & \leq & \left(\sum e^{-\lambda_{n}\text{Re}z}\varphi_{n}^{2}(x)\right)^{1/2}\left(\sum e^{-\lambda_{n}\text{Re}z}\varphi_{n}^{2}(x)\right)^{1/2}\\
 & \leq & \frac{C_{1}}{(\text{Re}z)^{\alpha}}.
\end{eqnarray*}
Therefore, we have $|u(z)|\leq\frac{C}{t^{\alpha}}$. On the other
hand, the kernel upper bound tells us that
\[
|u(z)|\leq\frac{C}{t^{\alpha}}\exp\left(-Cr^{\frac{\beta}{\beta-1}}z^{\frac{1}{\beta-1}}\right).
\]
Because of symmetry, we only prove the statement for the first quadrant.
Consider the strip $\Omega=\{x+iy:\ 0<y<\frac{\pi}{2}\}$ and let
\[
v(z)=u(\exp(z)).
\]
By assumption, $|v(z)|\leq\frac{C}{t^{\alpha}}$ for $z\in\overline{\Omega}$
and $|v(x)|\leq\frac{C}{t^{\alpha}}\exp(-C\exp(\alpha x))$ for $x\in\mathbb{R}$.
Let 
\[
f(z)=\frac{v(z)}{\frac{C}{t^{\alpha}}\exp(-Cr^{\frac{\beta}{\beta-1}}\exp(\frac{1}{\beta-1}z))}.
\]
Now $f$ is analytic on the strip $\Omega$, continuous on $\overline{\Omega}$
and bounded by $1$ on boundary of $\Omega$. Also, it satisfies a
decay estimate 
\[
|f(z)|\leq\exp(Cr^{\frac{\beta}{\beta-1}}\exp(\frac{1}{\beta-1}|z|))\quad(z\in\Omega).
\]
By the Phragmén-Lindelöf theorem, we have $|f|\leq1$ on $\Omega$.
Thus, for $z\in\{\text{Re}z>0\}$, we have

\begin{eqnarray*}
|u(z)| & \leq & \frac{C}{t^{\alpha}}|\exp(-Cr^{\frac{\beta}{\beta-1}}z^{\frac{1}{\beta-1}})|\\
 & \leq & \frac{C}{t^{\alpha}}\exp\left(-Cr^{\frac{\beta}{\beta-1}}z^{\frac{1}{\beta-1}}\sin(\frac{\pi}{2}-|\mbox{arg}z|)\right)\\
 & \leq & \frac{C}{t^{\alpha}}\exp\left(-Cr^{\frac{\beta}{\beta-1}}z^{\frac{2-\beta}{\beta-1}}\text{Re}z\right)
\end{eqnarray*}

\end{proof}
Recall that the heat equation is the averaged wave equation and we
can use this to recover the lower frequency of the solution of wave
equation. Therefore, mollified solutions is exponentially small outside
the support of initial function when time is small.

\begin{lem}
\label{lem:min_expoent}For $\frac{1}{4t}\geq\gamma>0$ and$-1<\alpha<0$,
we have
\[
\min_{\text{Re}(z)=\gamma}\left|\frac{1}{z^{-1}+t}\right|^{\alpha}\text{Re}(\frac{1}{z^{-1}+t})\geq C_{\alpha}t^{-\alpha/2}\gamma^{1+\alpha/2}
\]
where $C_{\alpha}$ is a constant depends on $\alpha$ only.\end{lem}
\begin{proof}
We may assume $\arg(\frac{1}{z^{-1}+t})\geq0$ because of symmetry.
For $0\leq\arg(\frac{1}{z^{-1}+t})\leq\frac{\pi}{4}$, we have 
\begin{eqnarray*}
\left|\frac{1}{z^{-1}+t}\right|^{\alpha}\text{Re}(\frac{1}{z^{-1}+t}) & \geq & \frac{1}{\sqrt{2}}\left|\frac{1}{z^{-1}+t}\right|^{\alpha+1}\\
 & \geq & \frac{1}{\sqrt{2}}\left|\frac{1}{\gamma^{-1}+t}\right|^{\alpha+1}
\end{eqnarray*}

For $\arg(\frac{1}{z^{-1}+t})>\frac{\pi}{4}$, let $z=\gamma+ix$.
We have 
\[
\frac{1}{z^{-1}+t}=\frac{\gamma^{2}+x^{2}}{(t(\gamma^{2}+x^{2})+\gamma)^{2}+x^{2}}(t(\gamma^{2}+x^{2})+\gamma+ix).
\]
Since $\arg(\frac{1}{z^{-1}+t})>\frac{\pi}{4}$, we have $x>t(\gamma^{2}+x^{2})+\gamma$.
Hence, 
\[
\frac{\gamma^{2}+x^{2}}{(t(\gamma^{2}+x^{2})+\gamma)^{2}+x^{2}}\geq\frac{1}{2}.
\]
Therefore, 
\begin{eqnarray*}
\left|\frac{1}{z^{-1}+t}\right|^{\alpha}\text{Re}(\frac{1}{z^{-1}+t}) & \geq & \left(\frac{1}{2}\right)^{\alpha+1}((t(\gamma^{2}+x^{2})+\gamma)^{2}+x^{2})^{\alpha/2}(t(\gamma^{2}+x^{2})+\gamma)\\
 & \geq & \left(\frac{1}{2}\right)^{\alpha+1}x^{\alpha}(tx^{2}+\gamma)\\
 & \geq & \left(\frac{1}{2}\right)^{\alpha+1}\left(\sqrt{\frac{-\alpha}{2+\alpha}}\right)^{\alpha}\frac{\gamma^{\alpha/2}}{t^{\alpha/2}}(\frac{2}{2+\alpha}\gamma)
\end{eqnarray*}
where the last line comes from minimizing $x$ over $x\geq0$. Combining
the two cases, we get 
\begin{eqnarray*}
\min_{\text{Re}(z)=\gamma}\left|\frac{1}{z^{-1}+t}\right|^{\alpha}\text{Re}(\frac{1}{z^{-1}+t}) & \geq & C_{\alpha}\min(\left|\frac{1}{\gamma^{-1}+t}\right|^{\alpha+1},\gamma^{1+\alpha/2}t^{-\alpha/2})\\
 & \geq & C_{\alpha}t^{-\alpha/2}\min(t^{-1},\gamma)^{1+\alpha/2}
\end{eqnarray*}
\end{proof}
\begin{thm}
\label{thm:avg_wave}Assume the heat kernel satisfies the upper bound
\eqref{eq:heat_assumption}. Let $\phi_{\sigma}(t)=\frac{1}{\sqrt{2\pi\sigma}}\exp(-\frac{t^{2}}{2\sigma})$.
For $\sigma=O\left(\left(\frac{t^{\beta-1}}{r^{\beta}}\right)^{\frac{1}{\beta-2}}\right)$
and $f\in L^{1}$, we have
\[
|(\phi_{\sigma}*Wf)(x,t)|\leq\frac{C}{t^{\alpha+1+\frac{5\beta-5}{\beta-2}}r^{\frac{5\beta}{\beta-2}}\sigma^{5/2}}\exp(-C\left(\frac{r^{2}}{t}\right)^{\frac{\beta}{\beta-2}}\sigma)||f||_{1}
\]
where $r=\text{d}(x,\text{supp}f)$. Furthermore, for $\sigma=O\left(\left(\frac{t^{2\beta-2}}{r^{2\beta}}\right)^{\frac{1}{\beta-2}}\right)$
and $P_{-\sigma}f\in L^{1}$ , we have
\[
|Wf(x,t)|\leq\frac{C}{t^{\alpha+1+\frac{5\beta-5}{\beta-2}}r^{\frac{5\beta}{\beta-2}}\sigma^{5/2}}\exp(-C\left(\frac{r^{2}}{t}\right)^{\frac{\beta}{\beta-2}}\sigma)||P_{-\sigma}f||_{1}
\]
where $r=\text{d}(x,\text{supp}f)$.\end{thm}
\begin{proof}
The relation between heat and wave equation can be written as
\[
\exp(-s\lambda_{n})=\sqrt{\frac{1}{4\pi s}}\int_{-\infty}^{\infty}\cos(t\sqrt{\lambda_{n}})e^{-\frac{t^{2}}{4s}}dt.
\]
Changing some variables, we get
\[
\sqrt{\frac{\pi}{s}}\exp\left(-\frac{\lambda_{n}}{4s}\right)=\int_{0}^{\infty}\frac{1}{\sqrt{t}}\cos(\sqrt{\lambda_{n}t})e^{-st}dt.
\]
The inverse Laplace transform implies

\begin{eqnarray*}
\frac{1}{\sqrt{t}}\cos(\sqrt{\lambda_{n}t}) & = & \frac{1}{2\pi i}\int_{\gamma-i\infty}^{\gamma+i\infty}\sqrt{\frac{\pi}{s}}\exp\left(-\frac{\lambda_{n}}{4s}+st\right)ds.
\end{eqnarray*}
for any $\gamma>0$. The mollified cosine is
\begin{eqnarray*}
\left(\phi*\cos(\sqrt{\lambda_{n}}\cdot)\right)(t) & = & e^{-\frac{\sigma}{2}\lambda_{n}}\cos(\sqrt{\lambda_{n}}t)\\
 & = & \frac{t}{2\pi i}\int_{\gamma-i\infty}^{\gamma+i\infty}\sqrt{\frac{\pi}{s}}\exp\left(-\lambda_{n}(\frac{1}{4s}+\frac{\sigma}{2})+st^{2}\right)ds.
\end{eqnarray*}
Hence, the mollified solution can be computed as
\[
(\phi_{\sigma}*Wf)(t)=\frac{t}{2\pi i}\int_{\gamma-i\infty}^{\gamma+i\infty}\sqrt{\frac{\pi}{s}}e^{st^{2}}P_{(4s)^{-1}+\frac{\sigma}{2}}fds.
\]
Rewrite the equation by letting $u(z)=P_{(4z)^{-1}+\frac{\sigma}{4}}f$
and $v(t)=(\phi_{\sigma}*Wf)(t)$, we have

\begin{eqnarray}
v(t) & = & \frac{t}{2\sqrt{\pi}i}\int_{\gamma-i\infty}^{\gamma+i\infty}s^{-1/2}e^{st^{2}}u(\frac{1}{s^{-1}+\sigma})ds\nonumber \\
 & = & \frac{1}{4\sqrt{\pi}it}\int_{\gamma-i\infty}^{\gamma+i\infty}s^{-3/2}e^{st^{2}}u(\frac{1}{s^{-1}+\sigma})ds+\nonumber \\
 &  & \frac{-\sqrt{\pi}}{2\pi it}\int_{\gamma-i\infty}^{\gamma+i\infty}s^{-1/2}e^{st^{2}}u'(\frac{1}{s^{-1}+\sigma})\frac{1}{(\frac{1}{s}+\sigma)^{2}}\frac{1}{s^{2}}ds.\label{eq:lemma14_eq}
\end{eqnarray}
Using Lemma \ref{lem:Phragm=0000E9n=002013Lindel=0000F6f}, we have
\begin{eqnarray*}
|u(\frac{1}{s^{-1}+\sigma})| & \leq & \frac{C}{t^{\alpha}}\exp\left(-Cr^{\frac{\beta}{\beta-1}}\left|\frac{1}{s^{-1}+\sigma}\right|^{\frac{2-\beta}{\beta-1}}\text{Re}\frac{1}{s^{-1}+\sigma}\right)||f||_{1}.
\end{eqnarray*}
Using Lemma \ref{lem:min_expoent}, for $\frac{1}{\sigma}>4\gamma>0$
, we have
\begin{eqnarray*}
|u(\frac{1}{s^{-1}+\sigma})| & \leq & \frac{C}{t^{\alpha}}\exp\left(-Cr^{\frac{\beta}{\beta-1}}\gamma^{\frac{\beta}{2\beta-2}}\sigma^{\frac{\beta-2}{2\beta-2}}\right)||f||_{1}
\end{eqnarray*}
for $\text{Re}s=\gamma$. By the Cauchy integral formula, we have
\begin{eqnarray*}
|u'(\frac{1}{s^{-1}+\sigma})| & \leq & \frac{C|s^{-1}+\sigma|}{t^{\alpha}}\exp\left(-Cr^{\frac{\beta}{\beta-1}}\gamma^{\frac{\beta}{2\beta-2}}\sigma^{\frac{\beta-2}{2\beta-2}}\right)||f||_{1}\\
 & \leq & \frac{C}{t^{\alpha}\gamma}\exp\left(-Cr^{\frac{\beta}{\beta-1}}\gamma^{\frac{\beta}{2\beta-2}}\sigma^{\frac{\beta-2}{2\beta-2}}\right)||f||_{1}
\end{eqnarray*}
for $\text{Re}s=\gamma$. Substituting in \eqref{eq:lemma14_eq},
we get 
\begin{eqnarray*}
|v(t)| & \leq & \frac{Ce^{\gamma t}}{t^{\alpha+1}}(\int_{0}^{\infty}|\gamma+is|^{-3/2}ds+\frac{1}{\gamma}\int_{0}^{\infty}|\gamma+is|^{-5/2}ds)\exp(-Cr^{\frac{\beta}{\beta-1}}\gamma^{\frac{\beta}{2\beta-2}}\sigma^{\frac{\beta-2}{2\beta-2}})\\
 & \leq & \frac{Ce^{\gamma t}}{t^{\alpha+1}}(\frac{1}{\sqrt{\gamma}}+\frac{1}{\gamma^{5/2}})\exp(-Cr^{\frac{\beta}{\beta-1}}\gamma^{\frac{\beta}{2\beta-2}}\sigma^{\frac{\beta-2}{2\beta-2}})
\end{eqnarray*}
The first result follows from putting $\gamma=\frac{\left(\frac{\beta}{2\beta-2}\right)^{\frac{\beta-2}{2\beta-2}}r^{\frac{2\beta}{\beta-2}}\sigma}{t^{\frac{2\beta-2}{\beta-2}}}$.
The second result follows from the identity 
\[
\left(\phi*\cos(\sqrt{\lambda_{n}}\cdot)\right)(t)=e^{-\frac{\sigma}{2}\lambda_{n}}\cos(\sqrt{\lambda_{n}}t).
\]

\end{proof}

\section{Acknowledgment}

I would like to thank the Department of Mathematics in the Chinese
University of Hong Kong and Cornell University providing me supports
and chance to do research via joining in the Research Experiences
for Undergraduates Program. I am grateful to Prof Robert Strichartz
for his valuable discussions and suggestions. Also, he read the original
manuscript and suggested several improvements. Lastly, I would like
to thank Ms. Mavis Chan for her valuable suggestions.

\bibliographystyle{amsplain}
\nocite{*}
\bibliography{fractal}

\providecommand{\bysame}{\leavevmode\hbox to3em{\hrulefill}\thinspace}
\providecommand{\MR}{\relax\ifhmode\unskip\space\fi MR }
\providecommand{\MRhref}[2]{%
  \href{http://www.ams.org/mathscinet-getitem?mr=#1}{#2}
}
\providecommand{\href}[2]{#2}
\begin{thebibliography}{10}

\bibitem{barlow1998diffusions}
M.~Barlow, \emph{Diffusions on fractals}, Lectures on probability theory and
  statistics (1998), 1--121.

\bibitem{barlow1988brownian}
M.T. Barlow and E.A. Perkins, \emph{Brownian motion on the sierpinski gasket},
  Probability theory and related fields \textbf{79} (1988), no.~4, 543--623.

\bibitem{dalrymple1999fractal}
K.~Dalrymple, R.S. Strichartz, and J.P. Vinson, \emph{Fractal differential
  equations on the sierpinski gasket}, Journal of Fourier Analysis and
  Applications \textbf{5} (1999), no.~2, 203--284.

\bibitem{davies1990heat}
E.B. Davies, \emph{Heat kernels and spectral theory}, vol.~92, Cambridge Univ
  Pr, 1990.

\bibitem{fitzsimmons1994transition}
P.J. Fitzsimmons, B.M. Hambly, and T.~Kumagai, \emph{Transition density
  estimates for brownian motion on affine nested fractals}, Communications in
  Mathematical Physics \textbf{165} (1994), no.~3, 595--620.

\bibitem{fukushima1992spectral}
M.~Fukushima and T.~Shima, \emph{On a spectral analysis for the sierpinski
  gasket}, Potential Analysis \textbf{1} (1992), no.~1, 1--35.

\bibitem{grigor2008off}
A.~Grigor¡Šyan and J.~Hu, \emph{Off-diagonal upper estimates for the heat
  kernel of the dirichlet forms on metric spaces}, Inventiones mathematicae
  \textbf{174} (2008), no.~1, 81--126.

\bibitem{hambly1999transition}
B.M. Hambly and T.~Kumagai, \emph{Transition density estimates for diffusion
  processes on post critically finite self-similar fractals}, Proceedings of
  the London Mathematical Society \textbf{78} (1999), no.~2, 431--458.

\bibitem{hu2002nonlinear}
J.~Hu, \emph{Nonlinear wave equations on a class of bounded fractal sets},
  Journal of mathematical analysis and applications \textbf{270} (2002), no.~2,
  657--680.

\bibitem{kigami1993harmonic}
J.~Kigami, \emph{Harmonic calculus on pcf self-similar sets}, american
  mathematical society \textbf{335} (1993), no.~2.

\bibitem{kigami1998distributions}
\bysame, \emph{Distributions of localized eigenvalues of laplacians on post
  critically finite self-similar sets}, Journal of functional analysis
  \textbf{156} (1998), no.~1, 170--198.

\bibitem{kigami2001analysis}
\bysame, \emph{Analysis on fractals}, no. 143, Cambridge Univ Pr, 2001.

\bibitem{rogers2009smooth}
L.G. Rogers, R.S. Strichartz, and A.~Teplyaev, \emph{Smooth bumps, a borel
  theorem and partitions of smooth functions on pcf fractals.}, Trans. Amer.
  Math. Soc \textbf{361} (2009), no.~4, 1765--1790.

\bibitem{shima1991eigenvalue}
T.~Shima, \emph{On eigenvalue problems for the random walks on the sierpinski
  pre-gaskets}, Japan Journal of Industrial and Applied Mathematics \textbf{8}
  (1991), no.~1, 127--141.

\bibitem{sikora2004riesz}
A.~Sikora, \emph{Riesz transform, gaussian bounds and the method of wave
  equation}, Mathematische Zeitschrift \textbf{247} (2004), no.~3, 643--662.

\bibitem{strichartz2006differential}
R.S. Strichartz, \emph{Differential equations on fractals: a tutorial},
  Princeton University Press, 2006.

\end{thebibliography}

\end{document}